\documentclass[12pt]{amsart}
\usepackage{amsmath,amssymb,amsbsy,amsfonts,amsthm,latexsym,
                        amsopn,amstext,amsxtra,euscript,amscd,mathrsfs,color,bm}
                        
\usepackage{float} 
\usepackage[english]{babel}
\usepackage{mathtools}
\usepackage[textsize=tiny]{todonotes}
\usepackage{url}
\usepackage[colorlinks,linkcolor=blue,anchorcolor=blue,citecolor=blue,backref=page]{hyperref}
\usepackage{nicefrac}

\usepackage[a4paper,  left=3.6cm, right=3.6cm, top=3.5 cm, bottom=3.5cm]{geometry}

\usepackage[norefs,nocites]{refcheck}

\begin{document}
%
    
\newtheorem{theorem}{Theorem}
\newtheorem{lemma}[theorem]{Lemma}
\newtheorem{example}[theorem]{Example}
\newtheorem{algol}{Algorithm}
\newtheorem{corollary}[theorem]{Corollary}
\newtheorem{prop}[theorem]{Proposition}
\newtheorem{proposition}[theorem]{Proposition}
\newtheorem{problem}[theorem]{Problem}
\newtheorem{conj}[theorem]{Conjecture}

\theoremstyle{remark}
\newtheorem{definition}[theorem]{Definition}
\newtheorem{question}[theorem]{Question}
\newtheorem{remark}[theorem]{Remark}
\newtheorem*{acknowledgement}{Acknowledgements}

\newtheorem*{Thm*}{Theorem}
\newtheorem{Thm}{Theorem}[section]
\renewcommand*{\theThm}{\Alph{Thm}}

\numberwithin{equation}{section}
\numberwithin{theorem}{section}
\numberwithin{table}{section}
\numberwithin{figure}{section}

\allowdisplaybreaks

\definecolor{olive}{rgb}{0.3, 0.4, .1}
\definecolor{dgreen}{rgb}{0.,0.5,0.}

\def\cc#1{\textcolor{red}{#1}} 

\definecolor{dgreen}{rgb}{0.,0.6,0.}
\def\tgreen#1{\begin{color}{dgreen}{\it{#1}}\end{color}}
\def\tblue#1{\begin{color}{blue}{\it{#1}}\end{color}}
\def\tred#1{\begin{color}{red}#1\end{color}}
\def\tmagenta#1{\begin{color}{magenta}{\it{#1}}\end{color}}
\def\tNavyBlue#1{\begin{color}{NavyBlue}{\it{#1}}\end{color}}
\def\tMaroon#1{\begin{color}{Maroon}{\it{#1}}\end{color}}

%


 \def\mand{\qquad\mbox{and}\qquad}

\def\cA{{\mathcal A}}
\def\cB{{\mathcal B}}
\def\cC{{\mathcal C}}
\def\cD{{\mathcal D}}
\def\cI{{\mathcal I}}
\def\cJ{{\mathcal J}}
\def\cK{{\mathcal K}}
\def\cL{{\mathcal L}}
\def\cM{{\mathcal M}}
\def\cN{{\mathcal N}}
\def\cO{{\mathcal O}}
\def\cP{{\mathcal P}}
\def\cQ{{\mathcal Q}}
\def\cR{{\mathcal R}}
\def\cS{{\mathcal S}}
\def\cT{{\mathcal T}}
\def\cU{{\mathcal U}}
\def\cV{{\mathcal V}}
\def\cW{{\mathcal W}}
\def\cX{{\mathcal X}}
\def\cY{{\mathcal Y}}
\def\cZ{{\mathcal Z}}

\def\C{\mathbb{C}}
\def\K{\mathbb{K}}
\def\Z{\mathbb{Z}}
\def\R{\mathbb{R}}
\def\Q{\mathbb{Q}}
\def\N{\mathbb{N}}
\def\M{\mathrm{M}}
\def\L{\mathbb{L}}
\def\M{{\normalfont\textsf{M}}} 
\def\U{\mathbb{U}}
\def\P{\mathbb{P}}
\def\A{\mathbb{A}}
\def\fp{\mathfrak{p}}
\def\fq{\mathfrak{q}}
\def\n{\mathfrak{n}}
\def\X{\mathcal{X}}
\def\x{\textrm{\bf x}}
\def\w{\textrm{\bf w}}
\def\ovQ{\overline{\Q}}
\def \Kab{\K^{\mathrm{ab}}}
\def \Qab{\Q^{\mathrm{ab}}}
\def \Qtr{\Q^{\mathrm{tr}}}
\def \Kc{\K^{\mathrm{c}}}
\def \Qc{\Q^{\mathrm{c}}}
\def\ZK{\Z_\K}
\def\ZKS{\Z_{\K,\cS}}
\def\ZKSf{\Z_{\K,\cS_{f}}}
\def\RSf{R_{\cS_{f}}}
\def\RTf{R_{\cT_{f}}}

\def\S{\mathcal{S}}
\def\vec#1{\mathbf{#1}}
\def\ov#1{{\overline{#1}}}
\def\sign{{\operatorname{sign}}}
\def\Gm{\G_{\textup{m}}}
\def\fA{{\mathfrak A}}
\def\fB{{\mathfrak B}}

\def \GL{\mathrm{GL}}
\def \Mat{\mathrm{Mat}}
\def \Tr{\mathrm{Tr}}

\def\house#1{{%
    \setbox0=\hbox{$#1$}
    \vrule height \dimexpr\ht0+1.4pt width .5pt depth \dp0\relax
    \vrule height \dimexpr\ht0+1.4pt width \dimexpr\wd0+2pt depth \dimexpr-\ht0-1pt\relax
    \llap{$#1$\kern1pt}
    \vrule height \dimexpr\ht0+1.4pt width .5pt depth \dp0\relax}}


\newenvironment{notation}[0]{%
  \begin{list}%
    {}%
    {\setlength{\itemindent}{0pt}
     \setlength{\labelwidth}{1\parindent}
     \setlength{\labelsep}{\parindent}
     \setlength{\leftmargin}{2\parindent}
     \setlength{\itemsep}{0pt}
     }%
   }%
  {\end{list}}

\newenvironment{parts}[0]{%
  \begin{list}{}%
    {\setlength{\itemindent}{0pt}
     \setlength{\labelwidth}{1.5\parindent}
     \setlength{\labelsep}{.5\parindent}
     \setlength{\leftmargin}{2\parindent}
     \setlength{\itemsep}{0pt}
     }%
   }%
  {\end{list}}
\newcommand{\Part}[1]{\item[\upshape#1]}

\def\Case#1#2{%
\smallskip\paragraph{\textbf{\boldmath Case #1: #2.}}\hfil\break\ignorespaces}

\def\Subcase#1#2{%
\smallskip\paragraph{\textit{\boldmath Subcase #1: #2.}}\hfil\break\ignorespaces}

\renewcommand{\a}{\alpha}
\renewcommand{\b}{\beta}
\newcommand{\g}{\gamma}
\renewcommand{\d}{\delta}
\newcommand{\e}{\epsilon}
\newcommand{\f}{\varphi}
\newcommand{\fhat}{\hat\varphi}
\newcommand{\bfphi}{{\boldsymbol{\f}}}
\renewcommand{\l}{\lambda}
\renewcommand{\k}{\kappa}
\newcommand{\lhat}{\hat\lambda}
\newcommand{\bfmu}{{\boldsymbol{\mu}}}
\renewcommand{\o}{\omega}
\renewcommand{\r}{\rho}
\newcommand{\rbar}{{\bar\rho}}
\newcommand{\s}{\sigma}
\newcommand{\sbar}{{\bar\sigma}}
\renewcommand{\t}{\tau}
\newcommand{\z}{\zeta}


\newcommand{\ga}{{\mathfrak{a}}}
\newcommand{\gb}{{\mathfrak{b}}}
\newcommand{\gn}{{\mathfrak{n}}}
\newcommand{\gp}{{\mathfrak{p}}}
\newcommand{\gP}{{\mathfrak{P}}}
\newcommand{\gq}{{\mathfrak{q}}}
\newcommand{\h}{{\mathfrak{h}}}
\newcommand{\Abar}{{\bar A}}
\newcommand{\Ebar}{{\bar E}}
\newcommand{\kbar}{{\bar k}}
\newcommand{\Kbar}{{\bar K}}
\newcommand{\Pbar}{{\bar P}}
\newcommand{\Sbar}{{\bar S}}
\newcommand{\Tbar}{{\bar T}}
\newcommand{\gbar}{{\bar\gamma}}
\newcommand{\lbar}{{\bar\lambda}}
\newcommand{\ybar}{{\bar y}}
\newcommand{\phibar}{{\bar\f}}

\newcommand{\Acal}{{\mathcal A}}
\newcommand{\Bcal}{{\mathcal B}}
\newcommand{\Ccal}{{\mathcal C}}
\newcommand{\Dcal}{{\mathcal D}}
\newcommand{\Ecal}{{\mathcal E}}
\newcommand{\Fcal}{{\mathcal F}}
\newcommand{\Gcal}{{\mathcal G}}
\newcommand{\Hcal}{{\mathcal H}}
\newcommand{\Ical}{{\mathcal I}}
\newcommand{\Jcal}{{\mathcal J}}
\newcommand{\Kcal}{{\mathcal K}}
\newcommand{\Lcal}{{\mathcal L}}
\newcommand{\Mcal}{{\mathcal M}}
\newcommand{\Ncal}{{\mathcal N}}
\newcommand{\Ocal}{{\mathcal O}}
\newcommand{\Pcal}{{\mathcal P}}
\newcommand{\Qcal}{{\mathcal Q}}
\newcommand{\Rcal}{{\mathcal R}}
\newcommand{\Scal}{{\mathcal S}}
\newcommand{\Tcal}{{\mathcal T}}
\newcommand{\Ucal}{{\mathcal U}}
\newcommand{\Vcal}{{\mathcal V}}
\newcommand{\Wcal}{{\mathcal W}}
\newcommand{\Xcal}{{\mathcal X}}
\newcommand{\Ycal}{{\mathcal Y}}
\newcommand{\Zcal}{{\mathcal Z}}

\renewcommand{\AA}{\mathbb{A}}
\newcommand{\BB}{\mathbb{B}}
\newcommand{\CC}{\mathbb{C}}
\newcommand{\FF}{\mathbb{F}}
\newcommand{\G}{\mathbb{G}}
\newcommand{\KK}{\mathbb{K}}
\newcommand{\NN}{\mathbb{N}}
\newcommand{\PP}{\mathbb{P}}
\newcommand{\QQ}{\mathbb{Q}}
\newcommand{\RR}{\mathbb{R}}
\newcommand{\ZZ}{\mathbb{Z}}

\newcommand{\bfa}{{\boldsymbol a}}
\newcommand{\bfb}{{\boldsymbol b}}
\newcommand{\bfc}{{\boldsymbol c}}
\newcommand{\bfd}{{\boldsymbol d}}
\newcommand{\bfe}{{\boldsymbol e}}
\newcommand{\bff}{{\boldsymbol f}}
\newcommand{\bfg}{{\boldsymbol g}}
\newcommand{\bfi}{{\boldsymbol i}}
\newcommand{\bfj}{{\boldsymbol j}}
\newcommand{\bfk}{{\boldsymbol k}}
\newcommand{\bfm}{{\boldsymbol m}}
\newcommand{\bfp}{{\boldsymbol p}}
\newcommand{\bfr}{{\boldsymbol r}}
\newcommand{\bfs}{{\boldsymbol s}}
\newcommand{\bft}{{\boldsymbol t}}
\newcommand{\bfu}{{\boldsymbol u}}
\newcommand{\bfv}{{\boldsymbol v}}
\newcommand{\bfw}{{\boldsymbol w}}
\newcommand{\bfx}{{\boldsymbol x}}
\newcommand{\bfy}{{\boldsymbol y}}
\newcommand{\bfz}{{\boldsymbol z}}
\newcommand{\bfA}{{\boldsymbol A}}
\newcommand{\bfF}{{\boldsymbol F}}
\newcommand{\bfB}{{\boldsymbol B}}
\newcommand{\bfD}{{\boldsymbol D}}
\newcommand{\bfG}{{\boldsymbol G}}
\newcommand{\bfI}{{\boldsymbol I}}
\newcommand{\bfM}{{\boldsymbol M}}
\newcommand{\bfP}{{\boldsymbol P}}
\newcommand{\bfX}{{\boldsymbol X}}
\newcommand{\bfY}{{\boldsymbol Y}}
\newcommand{\bfzero}{{\boldsymbol{0}}}
\newcommand{\bfone}{{\boldsymbol{1}}}

\newcommand{\aff}{{\textup{aff}}}
\newcommand{\Aut}{\operatorname{Aut}}
\newcommand{\Berk}{{\textup{Berk}}}
\newcommand{\Birat}{\operatorname{Birat}}
\newcommand{\characteristic}{\operatorname{char}}
\newcommand{\codim}{\operatorname{codim}}
\newcommand{\Crit}{\operatorname{Crit}}
\newcommand{\critwt}{\operatorname{critwt}} 
\newcommand{\cond}{\operatorname{cond}}
\newcommand{\Cycle}{\operatorname{Cycles}}
\newcommand{\diag}{\operatorname{diag}}
\newcommand{\Disc}{\operatorname{Disc}}
\newcommand{\Div}{\operatorname{Div}}
\newcommand{\Dom}{\operatorname{Dom}}
\newcommand{\End}{\operatorname{End}}
\newcommand{\ExtOrbit}{\mathcal{EO}} 
\newcommand{\Fbar}{{\bar{F}}}
\newcommand{\Fix}{\operatorname{Fix}}
\newcommand{\FOD}{\operatorname{FOD}}
\newcommand{\FOM}{\operatorname{FOM}}
\newcommand{\Gal}{\operatorname{Gal}}
\newcommand{\genus}{\operatorname{genus}}
\newcommand{\GITQuot}{/\!/}
\newcommand{\GR}{\operatorname{\mathcal{G\!R}}}
\newcommand{\Hom}{\operatorname{Hom}}
\newcommand{\Index}{\operatorname{Index}}
\newcommand{\Image}{\operatorname{Image}}
\newcommand{\Isom}{\operatorname{Isom}}
\newcommand{\hhat}{{\hat h}}
\newcommand{\Ker}{{\operatorname{ker}}}
\newcommand{\Ksep}{K^{\textup{sep}}}  
\newcommand{\lcm}{{\operatorname{lcm}}}
\newcommand{\LCM}{{\operatorname{LCM}}}
\newcommand{\Lift}{\operatorname{Lift}}
\newcommand{\limstar}{\lim\nolimits^*}
\newcommand{\limstarn}{\lim_{\hidewidth n\to\infty\hidewidth}{\!}^*{\,}}
\newcommand{\llog}{\log\log}
\newcommand{\logplus}{\log^{\scriptscriptstyle+}}
\newcommand{\maxplus}{\operatornamewithlimits{\textup{max}^{\scriptscriptstyle+}}}
\newcommand{\MOD}[1]{~(\textup{mod}~#1)}
\newcommand{\Mor}{\operatorname{Mor}}
\newcommand{\Moduli}{\mathcal{M}}
\newcommand{\Norm}{{\operatorname{\mathsf{N}}}}
\newcommand{\notdivide}{\nmid}
\newcommand{\normalsubgroup}{\triangleleft}
\newcommand{\NS}{\operatorname{NS}}
\newcommand{\onto}{\twoheadrightarrow}
\newcommand{\ord}{\operatorname{ord}}
\newcommand{\Orbit}{\mathcal{O}}
\newcommand{\Per}{\operatorname{Per}}
\newcommand{\Perp}{\operatorname{Perp}}
\newcommand{\PrePer}{\operatorname{PrePer}}
\newcommand{\PGL}{\operatorname{PGL}}
\newcommand{\Pic}{\operatorname{Pic}}
\newcommand{\Prob}{\operatorname{Prob}}
\newcommand{\Proj}{\operatorname{Proj}}
\newcommand{\Qbar}{{\bar{\QQ}}}
\newcommand{\rank}{\operatorname{rank}}
\newcommand{\Rat}{\operatorname{Rat}}
\newcommand{\Res}{{\operatorname{Res}}}
\newcommand{\Resultant}{\operatorname{Res}}
\renewcommand{\setminus}{\smallsetminus}
\newcommand{\sgn}{\operatorname{sgn}}
\newcommand{\SL}{\operatorname{SL}}
\newcommand{\Span}{\operatorname{Span}}
\newcommand{\Spec}{\operatorname{Spec}}
\renewcommand{\ss}{{\textup{ss}}}
\newcommand{\stab}{{\textup{stab}}}
\newcommand{\Stab}{\operatorname{Stab}}
\newcommand{\Support}{\operatorname{Supp}}
\newcommand{\Sym}{\operatorname{Sym}}  
\newcommand{\tors}{{\textup{tors}}}
\newcommand{\Trace}{\operatorname{Trace}}
\newcommand{\trianglebin}{\mathbin{\triangle}} 
\newcommand{\tr}{{\textup{tr}}} 
\newcommand{\UHP}{{\mathfrak{h}}}    
\newcommand{\Wander}{\operatorname{Wander}}
\newcommand{\<}{\langle}
\renewcommand{\>}{\rangle}

\newcommand{\pmodintext}[1]{~\textup{(mod}~#1\textup{)}}
\newcommand{\ds}{\displaystyle}
\newcommand{\longhookrightarrow}{\lhook\joinrel\longrightarrow}
\newcommand{\longonto}{\relbar\joinrel\twoheadrightarrow}
\newcommand{\SmallMatrix}[1]{%
  \left(\begin{smallmatrix} #1 \end{smallmatrix}\right)}
  
  \def\({\left(}
\def\){\right)}


\title[]
{On a Problem of Lang for Matrix Polynomials}

\author[A. Ostafe] {Alina Ostafe}
\address{School of Mathematics and Statistics, University of New South Wales, Sydney NSW 2052, Australia}
\email{alina.ostafe@unsw.edu.au}

\begin{abstract} 
In this paper, we consider a problem of Lang about finiteness of torsion points on plane rational curves, and prove some results towards a matrix analogue  of this problem.  
\end{abstract}

\maketitle

\section{Introduction and statements of main results}
\subsection{Motivation} 

 Pivotal work of Lang made it clear that the existence of multiplicative relations between coordinates of points on algebraic curves in $\G_m^n=(\C\setminus\{0\})^n$ is a very rare event, which may occur only if the curve is ``special''. In particular, the celebrated result conjectured by  Lang~\cite{Lang, Zan}  in the 1960s and proved by Ihara, Serre and Tate  asserts  the finiteness of so-called {\it torsion points} on curves, that is, points with all coordinates roots of unity.  For the case of plane curves, Beukers and Smyth~\cite[Section~4.1]{BS} give a uniform bound for the number of such points, and Corvaja and Zannier~\cite{CZ08} give an upper bound  for the maximal order of torsion points  on the curve. More precisely, one has the following result~\cite[Section~4.1]{BS}:

\begin{Thm}
\label{thm:Lang}
An algebraic curve $F(y_1,y_2)=0$, where $F\in\C[y_1,y_2]$, contains at most $11(\deg F)^2$ torsion points unless $F$ has a factor of the form  $y_1^i-\rho y_2^j$ or $y_1^iy_2^j-\rho$
for some nonnegative integers $i,j$ not both zero and some root of unity $\rho$.
\end{Thm}

Theorem~\ref{thm:Lang} in the case of plane rational curves can be reformulated as follows: given multiplicatively independent rational functions $f,g\in\C(x)$ (see below for the precise definition), there are at most
$$11(\deg f+\deg g)^2\min(\deg f,\deg g)\le 22 (\deg f +\deg g)\deg f \cdot \deg g$$ 
elements $\alpha\in\C$ such that both $f(\alpha)$ and $g(\alpha)$ are roots of unity, see also the proof of~\cite[Lemma 2.2]{Ost}. This has been extended to a finiteness result of elements $\alpha\in\C$ such that $|f(\alpha)|=|g(\alpha)|=1$, first by Corvaja, Masser and Zannier~\cite{CMZ} for $f(x)=x$ and $g\in\C[x]$, and later by Pakovich and Shparlinski~\cite{PakShp} for the general case, improving also the bound above for genus zero curves. More precisely, we have the following result~\cite[Theorem~2.2]{PakShp}:

\begin{Thm}
\label{thm:PakShp}
Let $f,g\in\C(x)$. Then one has
$$
\#\{\alpha\in \C : |f(\alpha)|=|g(\alpha)|=1\}\le (\deg f+\deg g)^2,
$$
unless 
$$
f=f_1\circ h \mand g=g_1\circ h
$$
for some quotients of Blaschke products $f_1$ and $f_2$ and some rational function $h$.
\end{Thm}

As remarked in~\cite{PakShp} (see the comment after Theorem 2.2 in~\cite{PakShp}), if $f$ and $g$ are polynomials, then the conclusion of Theorem~\ref{thm:PakShp} holds, unless the polynomials $f$ and $g$ are multiplicatively dependent.

\medskip
In this note, we obtain some results towards an analogue of Theorem~\ref{thm:Lang} (for plane rational curves) for matrix polynomials. 

\medskip
{\it Notation and conventions:}
We now set the following notation, which remains fixed for the remainder of
this paper:
\begin{notation}
  \item[\textbullet] For $r\ge 1$, $\M_r(\C)$ is the set of all $r\times r$ matrices with entries in $\C$, $\GL_r(\C)$ the set of invertible matrices, and $\SL_r(\C)$ the set of matrices of determinant one.
    \item[\textbullet]  $I\in \M_r(\C)$ is the identity matrix.
         \item[\textbullet] We use $0$ for both  the zero scalar and the zero matrix, which shall be clear from the context.
   \item[\textbullet] By a scalar matrix we mean  a  scalar multiple of the identity $I$, that is, $\lambda I$ for some $\lambda\in\C$. 
  \item[\textbullet] $x,y_1,y_2$ are ``scalar" variables, that is, we  apply them at elements $\lambda\in \C$. We reserve $Z, Z_1, Z_2$ for  ``matrix" variables, that is, we apply them at matrices $A\in\M_r(\C)$.
  
   \noindent We also write  $xI$ for the multiplication of the variable $x$ with the identity matrix $I$.
 \item[\textbullet]  By a matrix polynomial $f\in \M_r(\C)[Z]$  with coefficients in $\M_r(\C)$ we mean a polynomial of the form 
 $$
 f(Z)=C_dZ^d+\cdots+C_1Z+C_0,\qquad C_i\in \M_r(\C), \quad i=0,\ldots,d,
 $$
 for some $d\ge 1$ with $C_d\ne 0$.

 \item[\textbullet] For $A\in \M_r(\C)$, we write $A^T$ for the transpose of $A$.
   \item[\textbullet] For $A\in \M_r(\C)$, $\det(A)$ is the determinant of the matrix $A$, $\Tr(A)$ is the trace of $A$ and $\Spec(A)$ is the set of its eigenvalues.
  \item[\textbullet] $A\in \GL_r(\C)$ is called {\it torsion matrix} if $A^n=I$ for some $n\ge 1$. A pair of matrices $(A,B)$ is called a {\it torsion point} in $\GL_r(\C)^2$ if both matrices $A$ and $B$ are torsion. 
\end{notation}

We say that two matrices $A,B\in \M_r(\C)$ are {\it conjugate} if there exists an invertible matrix $V\in \M_r(\C)$ such that
$$
A=VBV^{-1}.
$$
Clearly, two conjugate matrices have the same set of eigenvalues with the same multiplicities.
We also recall that  a conjugacy class $\cA$ containing an element $A\in \M_2(\C)$ is the set of all matrices of the form $UAU^{-1}$, $U\in\GL_2(\C)$.

We say that two algebraic functions $h_1,h_2\in\ov{\C(x)}$ are {\it multiplicatively dependent}  if there is a non-zero vector $(k_1,k_2)\in\Z^2$ such that 
$$
h_1(x)^{k_1}h_2(x)^{k_2} =1.
$$
Otherwise they are called {\it multiplicatively independent}.

\medskip

As a direct consequence of Theorem~\ref{thm:PakShp}, one already has an immediate result for matrix polynomials $f,g\in \M_r(\C)[Z]$ such that all the eigenvalues of $f(\lambda I)$ and $g(\lambda I)$, $\lambda\in \C$, are of absolute value one. More precisely, one has:

\begin{corollary}
\label{cor:eigen abs val 1}
Let $f,g\in \M_r(\C)[Z]$ be such that $\det(f(x I))$ and $\det(g(x I))$ are multiplicatively independent in $\C(x)$.  Then there are at most 
$$
r^2(\deg f+\deg g)^2
$$
elements $\lambda\in\C$ such that $f(\lambda I)$ and $g(\lambda I)$ satisfy
$$
|\det(f(\lambda I))|=|\det(g(\lambda I))|=1.
$$
In particular, there are at most finitely many elements $\lambda\in\C$ such that all eigenvalues of $f(\lambda I)$ and $g(\lambda I)$ are of absolute value one.
\end{corollary}

\begin{remark}

The %
condition that $\det(f(x I))$ and $\det(g(x I))$ are multiplicatively independent in $\C(x)$ in Corollary~\ref{cor:eigen abs val 1} can be reformulated as follows:
there is no non-zero vector $(k_1,k_2)\in\Z^2$ such that
$$
f(x I)^{k_1} g(x I)^{k_2} \in \SL_r(\C(x)).
$$
Indeed, $\det(f(x I))$ and $\det(g(x I))$ are multiplicatively independent in $\C(x)$ if and only if there is no non-zero vector $(k_1,k_2)\in\Z^2$ such that 
$$
\det(f(x I))^{k_1} \det(g(x I))^{k_2} =\det \(f(x I)^{k_1}g(x I)^{k_2}\)=1,
$$
which implies the above condition.
\end{remark}

\begin{remark}
\label{rem:scalar}
We also note that if $f,g\in \C[Z]$, then for any matrix $A\in \M_r(\C)$, by the spectral theorem on 
eigenvalues, the eigenvalues of $f(A)$ are $f(\lambda_i)$, $i=1,\ldots,r$, where $\lambda_1,\ldots,\lambda_r$ are the eigenvalues of $A$, and similarly for $g$. Thus, if $f(A)^n=I$ for some $n$, then all $f(\lambda_i)$, $i=1,\ldots,r$, are roots of unity, and similarly for $g$. We reduce thus the problem to the classical Lang problem,  that is, Theorem~\ref{thm:Lang}. Similarly, if all eigenvalues of $f(A)$ and $g(A)$ are of absolute value one, then we reduce the problem to Theorem~\ref{thm:PakShp}.

If $f,g\in \M_r(\C)[Z]$ with coefficients $C_i=c_iI$, $i=1,\ldots,\deg f$, and similarly for  $g$, then we are in the case above, that is, $f\in \C[Z]$ is given by
$$
f(Z)=\sum_{i=0}^{\deg f} c_i Z^i,
$$
and similarly for $g$, and thus the discussion above applies, again.
\end{remark}

\medskip
Theorem~\ref{thm:Lang} is also intimately related to the question of giving uniform bounds for the degree of $\gcd(f^n-1,g^m-1)$, $n,m\ge 1$, for some polynomials $f,g\in\C[x]$, which was initially considered by Ailon and Rudnick~\cite{AR} and later in~\cite{Ost} and further extended in several ways by other authors. It is worth mentioning that matrices have already been considered in this context in~\cite{AR}, that is, the authors give results for $\gcd(A^n-I)$, $n\ge 1$, for a matrix $A$ defined over $\Z$, cyclotomic extensions or $\C[T]$ (here, by the greatest common divisor of a matrix we mean the greatest common divisor of all entries of the matrix). Moreover, in~\cite{CoRuZa}, Corvaja, Rudnick and Zannier study the growth of the order of matrices in reduction modulo integers $N\ge 1$ as $N$ goes to infinity.

We note that the finiteness result in Theorem~\ref{thm:Lang} has been extended to higher order multiplicative relations of points on curves in $\G_m^n$ defined over $\ov\Q$ by Bombieri, Masser and Zannier~\cite{BMZ99}, and then further generalised in~\cite{BMZ08,Mau}.


\subsection{Main results}
Informally, given matrix polynomials  $f,g\in \M_r(\C)[Z]$, we would like to understand the presence of matrices $A\in \M_r(\C)$, such that $f(A)$ and $g(A)$ are ``roots'' of the identity matrix. In this paper, we are able to prove a finiteness result in any dimension $r\ge 2$ for the set of such specialisations $A\in \M_r(\C)$ that commute with the coefficients of both $f$ and $g$, as well as for arbitrary  matrices $A\in \M_{2}(\C)$ in dimension two when $f(Z)=Z$ and $g(Z)=Z^d+C$ for some fixed $C\in\M_2(C)$. The latter follows from a necessary and sufficient characterisation of torsion solutions to the equation
$$
Z_1+Z_2=C.
$$

It is clear that, in the case of matrices, one cannot expect a finiteness result as in Theorem~\ref{thm:Lang}. For example, let $f$ have  the coefficients  $c_iI$, $c_i\in\C$, $i=0,\ldots,\deg f$, and let $A\in \M_r(\C)$ be such that $f(A)^n=I$ for some $n$. Then any  matrix conjugate to $A$  is also a solution to $f(Z)^n=I$, and similarly for $g$. Thus, one can only expect a finiteness result {\it up to conjugacy}.

\medskip
Our first result gives an answer towards Lang's problem for matrices which commute with the coefficients of the polynomials $f$ and $g$. More precisely, we have:

\begin{theorem}
\label{thm:LangMatr comm}
Let $f,g\in \M_r(\C)[Z]$ be 
such that 
any eigenvalue of $f(x I)$ and any eigenvalue of $g(x I)$ are multiplicatively independent functions in $\ov{\C(x)}$.  Then, up to conjugacy, there are at most 
$$
2\(22 r^{5} (\deg f+\deg g)(\deg f\cdot \deg g)\)^r
$$
matrices $A\in \M_r(\C)$ which commute with the coefficients of $f$ and $g$, such that $(f(A),g(A))$ is a torsion point in $\GL_r(\C)^2$.
\end{theorem}

The proof reduces to considering scalar specialisations, see Lemma~\ref{lem:eigen root 1} (in Section~\ref{sec:scalar}), and thus relies on Theorem~\ref{thm:Lang} above.

As an example, one can consider all coefficients of $f$ and $g$ to be matrices in  $\C[B]$ 
for some  fixed $B \in \M_r(\C)$. Then 
Theorem~\ref{thm:LangMatr comm} gives finiteness, up to conjugacy, of the set of matrices $A\in\M_r(\C)$ which commute with $B$, such that $(f(A),g(A))$ is a torsion point.

\medskip
Our second main result removes the commutativity condition on the specialisations, but applies only to certain linear polynomials. More precisely, we obtain the following characterisation of torsion solutions to linear equations of the form $Z_1+Z_2=C$, $C\in\M_2(\C)$.

\begin{theorem}
\label{thm:lang deg 1}
Let $C \in \M_2(\C)$. 
\begin{itemize}
\item[(i)] Assume $C\ne \mu\cdot I$ for some $\mu\in\C$. If $\Tr(C)$ is not the sum of at most two roots of unity, then there are only finitely many pairs of conjugacy classes $(\cA,\cB)\subset \M_2(\C)^2$ that contain the torsion matrix solutions to the equation
$$
Z_1+Z_2=C.
$$
Conversely, if there are only finitely many pairs of conjugacy classes $(\cA,\cB)\subset \M_2(\C)^2$ that contain the torsion matrix solutions to the equation
$$
Z_1+Z_2=C,
$$
then $\Tr(C)$ is not the sum of at most two roots of unity.


\item[(ii)]  If $C=\mu\cdot I$ for some $\mu\in\C^*$, then there are only finitely many conjugacy classes $\cA$  that contain torsion matrices $A \in\GL_2(\C)$ such that $\mu\cdot I-A$ is also torsion.
\end{itemize}

\end{theorem}

We conclude  with the following consequence. 

\begin{corollary}
\label{cor:Z^d+c}
Let $C\in\M_2(\C)$ be such that $\Tr(C)$ is not the sum of at most two roots of unity, and let $f(Z)=Z^d+C\in\M_2(\C)[Z]$ be a polynomial of degree $d\ge 1$. Then, up to conjugacy, there are only finitely many torsion matrices $U$ such that $f(U)$ is also torsion.
\end{corollary}

\section{Preliminaries}

\subsection{Multiplicative independence of eigenvalues}

Let $f,g\in \M_r(\C)[Z]$. We  define
\begin{equation}
\label{eq:Pfg}
\begin{split}
&P_f(x,y_1)=\det\(y_1 I -f(x I)\)\in\C[x,y_1],\\
&P_g(x,y_2)=\det\(y_2 I -g(x I)\) \in\C[x,y_2],
 \end{split}
\end{equation}
and the resultant
\begin{equation}
\label{eq:R}
\begin{split}
&R_{f,g}(y_1,y_2)=\Res_x\(P_f(x,y_1),P_g(x,y_2)\) \in\C[y_1,y_2].
 \end{split}
\end{equation}

We note that $R_{f,g}$ is a non-zero polynomial. Indeed, assume that $R_{f,g}=0$. Then, 
by the definition of the resultant,  the polynomials $P_f(x,y_1)$ and $P_g(x,y_2)$, as polynomials in $x$, share a common root $t\in \ov{\C(y_1)}\cap \ov{\C(y_2)}=\C$. Thus we obtain that $\det(y_1I-f(tI))=\det(y_2I-g(tI))=0$, which is a contradiction, since both polynomials have as leading monomials $y_1^r$ and $y_2^r$, respectively. 

We know that $\deg_x P_f\le r\deg f$ and $\deg_x P_g\le r\deg g$, and $R_{f,g}$ is a polynomial of degree $\deg_x P_f$ in $y_2$ and of degree $\deg_x P_g$ in $y_1$.  We thus obtain that
\begin{equation}
\label{eq:deg R}
\deg R_{f,g}\le r(\deg f+\deg g).
\end{equation}

\begin{lemma}
\label{lem:m dep cond} 
Let $f,g\in \M_r(\C)[Z]$. If any eigenvalue of $f(x I)$ and any eigenvalue of $g(x I)$ are multiplicatively independent functions in $\ov{\C(x)}$,
then
$R_{f,g}(y_1,y_2)$ defined by~\eqref{eq:R} does not have a factor of the form $y_1^iy_2^j-\rho$ or $y_1^i-\rho y_2^j$ for some non-negative integers $i,j$ not both zero and some root of unity $\rho$.
\end{lemma}

\begin{proof} Let $\mu_i(x)$, $i=1,\ldots,r$, be the eigenvalues of $f(x I)$ in $\ov{\C(x)}$, that is, the roots of the polynomial $P_f(y_1,x)$ defined by~\eqref{eq:Pfg} as a polynomial in $y_1$. Similarly, let $\eta_j(x)$, $j=1,\ldots,r$, be the eigenvalues of $g(x I)$ in $\ov{\C(x)}$.

Assume that $R_{f,g}(y_1,y_2)$ has  a factor of one of the forbidden forms, say $y_1^iy_2^j-\rho$ for some non-negative integers $i,j$ not both zero and some root of unity $\rho$. We note that any point on the curve $R_{f,g}(y_1,y_2)=0$ is of the form $(\mu_k(t),\eta_\ell(t))$ for some $1\le k,\ell\le r$ and some $t\in\C$. Indeed, let $(t_1,t_2)\in\C^2$ be such that $R_{f,g}(t_1,t_2)=0$. Then, by definition of the resultant $R_{f,g}$, the two polynomials
$$\det(t_1 I-f(x I))=\prod_{i=1}^r(t_1-\mu_i(x)),\ \det(t_2 I-g(x I))=\prod_{i=1}^r(t_2-\eta_i(x))$$
have a common root $x=t\in \C$. This implies that $t_1=\mu_k(t)$ and $t_2=\eta_\ell(t)$ for some $k,\ell$. Since $y_1^iy_2^j-\rho$ is a factor of $R_{f,g}$, there are infinitely many $(t_1,t_2)\in\C^2$ which are roots of this factor, and thus we deduce that there are infinitely many $t\in\C$ such that
$$
\mu_k(t)^{i}\eta_\ell(t)^{j}=\rho
$$
for some $1\le k,\ell\le r$. 
Since $\mu_k$ and $\eta_\ell$ are algebraic functions, we conclude that $\mu_k(x)^{i}\eta_\ell(x)^{j}=\rho$, which contradicts our hypothesis. 

The case when $R_{f,g}(y_1,y_2)$ has a factor of the form $y_1^i-\rho y_2^j$ is treated entirely similar. 
\end{proof}

\subsection{Scalar specialisations}
\label{sec:scalar}
The main tool for the proof of Theorem~\ref{thm:LangMatr comm} is the following result which applies, again, to scalar matrices $\lambda I$, however for which the matrices $f(\lambda I)$ and $g(\lambda I)$ satisfy  different conditions than in Corollary~\ref{cor:eigen abs val 1}. More precisely, we have:
\begin{lemma}
\label{lem:eigen root 1}
Let $f,g\in \M_r(\C)[Z]$ be such that any eigenvalue of $f(x I)$ and any eigenvalue of $g(x I)$ are multiplicatively independent functions in $\ov{\C(x)}$. 
 Then there are at most 
$$
22 r^{5} (\deg f+\deg g)\deg f\cdot \deg g
$$
elements $\lambda\in\C$ such that 
$$
f(\lambda I)^n-I \mand g(\lambda I)^m-I
$$
are singular matrices for some $n,m\ge 1$.
\end{lemma}

\begin{proof}
We use a similar approach as for the proof of~\cite[Theorem 3]{AR}, reducing the problem to an application of Theorem~\ref{thm:Lang}.

Let $\lambda\in \C$ be such that $f(\lambda I)^n-I$ and $g(\lambda I)^m-I$
are singular matrices for some $n,m\ge 1$. This implies that $J_{f(\lambda I)}^n$ and $J_{g(\lambda I)}^m$, which are triangular matrices, have at least one element $1$ on the main diagonal, where $J_{f(\lambda I)}$ and $J_{g(\lambda I)}$ are Jordan normal forms of $f(\lambda I)$ and $g(\lambda I)$, respectively.

Let $u_{\lambda,i},v_{\lambda,j}\in\C$, $i,j=1,\ldots,r$, be the eigenvalues of $f(\lambda I),g(\lambda I)$, respectively, that is, $u_{\lambda,i}$ are the (not necessarily distinct) roots of the polynomial $P_f(\lambda,y_1)$ and $v_{\lambda,j}$ are the (not necessarily distinct) roots of the polynomial $P_g(\lambda,y_2)$, where $P_f(x,y_1)$ and $P_g(x,y_2)$ are defined by~\eqref{eq:Pfg}. Consequently, there exist $i,j\in\{1,\ldots,r\}$ such that $u_{\lambda,i}^n=1$ and $v_{\lambda,j}^m=1$,  that is, both $u_{\lambda,i}$ and $v_{\lambda,j}$ are roots of unity.

Notice 
that, since $P_f(\lambda,u_{\lambda,i})=P_g(\lambda,v_{\lambda,j})=0$, one also has $$R_{f,g}(u_{\lambda,i},v_{\lambda,j})=0$$ for all $i,j$, where $R_{f,g}$ is defined by~\eqref{eq:R}. Therefore, from the above discussion, there exist $i,j$ such that $(u_{\lambda,i},v_{\lambda,j})$ is a torsion point on the curve $R_{f,g}(y_1,y_2)=0$.

Since, by our hypothesis and Lemma~\ref{lem:m dep cond}, $R_{f,g}$ does not have any of the special factors mentioned in the statement of 
Theorem~\ref{thm:Lang}, 
it follows from Theorem~\ref{thm:Lang} and~\eqref{eq:deg R} that there are at most 
$$
11(\deg R_{f,g})^2\le 11 r^2 (\deg f+\deg g)^2
$$
torsion points $(\zeta_1,\zeta_2)$ on the curve $R_{f,g}(y_1,y_2)=0$. Each such point $(\zeta_1,\zeta_2)=(u_{\lambda,i},v_{\lambda,j})$ for some $i,j$ corresponds to at most $r\min(\deg f, \deg g)$ values of $\lambda$. Indeed, since $R_{f,g}(\zeta_1,\zeta_2)=0$,  $\lambda$ is a common root of the polynomials $P_f(x,\zeta_1),P_g(x,\zeta_2)$. We note that both polynomials $P_f(x,\zeta_1),P_g(x,\zeta_2)$ are non-zero, since, otherwise, $\zeta_1$ or $\zeta_2$ would be an   eigenvalue of $f(xI)$ or $g(xI)$, respectively. However, since $\zeta_1$ or $\zeta_2$  are roots of unity, this contradicts the multiplicative independence assumption on the eigenvalues of $f(xI)$ and $g(xI)$.

Taking the contribution from each $i,j\le r$, we conclude that there at most
$$
11 r^{5} (\deg f+\deg g)^2\min(\deg f,\deg g)\le 22r^{5}(\deg f+\deg g)\deg f\cdot \deg g
$$
 possibilities for such $\lambda\in \C$,  which concludes the proof.
\end{proof}

\begin{remark}
\label{rem:lang scalar}
It is worth mentioning that Lemma~\ref{lem:eigen root 1} is equivalent to the following reformulation:

Let $f,g\in \M_r(\C)[Z]$ be as in Lemma~\ref{lem:eigen root 1}. Then there are at most
$$
22 r^{5} (\deg f+\deg g)\deg f\cdot \deg g
$$
elements $\lambda\in \C$ such that
$f(\lambda I)$ and $g(\lambda I)$ have each at least one eigenvalue that is a root of unity.
\end{remark}

\begin{remark}
When $r=1$, the conditions in Corollary~\ref{cor:eigen abs val 1} and Lemma~\ref{lem:eigen root 1}  are equivalent  to the polynomials $f$ and $g$ being multiplicatively independent,  and, in this case, we recover Theorem~\ref{thm:Lang}.
\end{remark}

\subsection{On sums of roots of unity}

The main tool for the proof of Theorem~\ref{thm:lang deg 1} is the following result regarding linear relations of roots of unity, due initially to Mann~\cite{Mann} for equations over $\Q$ and further extended to linear equations with algebraic coefficients in~\cite[Theorem 1]{DZ} (see also previous work in~\cite{Zan95}). 

\begin{lemma}
\label{lem:sum roots 1}
Let $a_i\in\ov\Q^*$, $i=0,\ldots,s$. Then the equation
\begin{equation}
\label{eq:sum}
\sum_{i=1}^s a_iX_i=a_0
\end{equation}
has only finitely many non-degenerate solutions in roots of unity, that is, solutions in roots of unity $(\zeta_1,\ldots,\zeta_s)$ for which there is no proper  subsum in~\eqref{eq:sum} that vanishes.
\end{lemma}

It should be also noted, see~\cite[Theorem 1]{DZ}, that the number of non-degenerate solutions in roots of unity to the equation~\eqref{eq:sum} can be bounded only in terms of $s$ and the degree of the number field containing the coefficients $a_0,\ldots,a_s$.

\section{Proofs of main results}

\subsection{Proof of Theorem~\ref{thm:LangMatr comm}}
The proof follows as a simple application of Lemma~\ref{lem:eigen root 1}. Indeed, let $A\in \M_r(\C)$ be such that $A$ commutes with each of the coefficients of $f$ and $g$, and such that 
\begin{equation}
\label{eq:f g n m}
f(A)^n=I \mand g(A)^m=I
\end{equation}
for some $n,m\ge 1$. 

Using the commutativity assumption on $A$, simple computations show that there exist polynomials $Q_{n,A}, Q_{m,A}\in \M_r(\C)$ depending on $n,m$ and $A$, such that
\begin{eqnarray*}
&f(x I)^n-f(A)^n=Q_{n,A}(x I)(x I-A),\\
&g(x I)^m-g(A)^m=Q_{m,A}(x I)(x I-A).
\end{eqnarray*}
Therefore, using~\eqref{eq:f g n m}, we obtain that
$$
\det(x I-A) \mid \gcd\(\det(f(x I)^n-I), \det(g(x I)^m-I)\).
$$ 

We note that both polynomials $\det(f(x I)^n-I)$ and $\det(g(x I)^m-I)$ are non-zero. Indeed, assume, for example, that $\det(f(x I)^n-I)=0$. Then writing
$$
\det(f(x I)^n-I)=\prod_{i=1}^n \det(f(x I)-\zeta^iI),
$$
where $\zeta\in\C$ is an $n$-th root of unity, we conclude that $\det(f(x I)-\zeta^iI)=0$ for some $i=1,\ldots,n$. Thus $\zeta^i$ is an eigenvalue of $f(xI)$, and similarly for $g$. This contradicts our multiplicative independence assumption on the eigenvalues of $f(xI)$ and $g(xI)$.

Thus, every eigenvalue of $A$ is a root of the greatest common divisor above. In other words, for any eigenvalue $\lambda\in\C$ of $A$, the matrices $f(\lambda I)^n-I$ and $g(\lambda I)^m-I$ are singular. The conclusion now follows from Lemma~\ref{lem:eigen root 1}, that is, there are at most 
$$
L=22 r^{5} (\deg f+\deg g)\deg f\cdot \deg g
$$
possibilities for each of the eigenvalues of $A$. 

We partition now the set $\{1,\ldots,r\}$ into $k$ ordered parts, $1\le k\le r$, where each such part corresponds to a Jordan block of $A$, and thus to one eigenvalue $\lambda$.  The number of such partitions is $\binom{r-1}{k-1}$, and each set in a partition corresponds to at most $L$ values of $\lambda\in \C$. Summing over all $k$ we obtain at most
$$
\sum_{k=1}^r \binom{r-1}{k-1}L^k=L\sum_{k=0}^{r-1} \binom{r-1}{k}L^{k}=L(L+1)^{r-1}\le L^r(1+1/L)^{L/2}\le 2L^r
$$
possible Jordan normal forms, which concludes the proof.

\subsection{Proof of Theorem~\ref{thm:lang deg 1}}
We can write $C=VDV^{-1}$, where $V\in\GL_2(\C)$ and 
\begin{equation}
\label{eq:D}
D=\begin{pmatrix}
\mu_1 & 0\\
0 & \mu_2
\end{pmatrix} \qquad \textrm{or} \qquad D=\begin{pmatrix}
\mu & 1\\
0 & \mu
\end{pmatrix},
\end{equation}
where $\mu_1,\mu_2,\mu$ are the eigenvalues of $C$.

We look for torsion solutions $(A,B)\in\GL_2(\C)^2$ to the equation $Z_1+Z_2=C$. Since $A,B$ are torsion if and only if $V^{-1}AV$, $V^{-1}BV$ are as well, the problem is equivalent to looking at torsion solutions to the equation
\begin{equation}
\label{eq:xyd}
Z_1+Z_2=D,
\end{equation}
where $D$ has one of the forms in~\eqref{eq:D}.

We first remark that if $D$ is diagonal in~\eqref{eq:D}, and $\mu_1=\mu_2$, that is, $C=\mu_1\cdot I$, then (ii) follows directly from Remark~\ref{rem:scalar} or Theorem~\ref{thm:LangMatr comm} (since any specialisation will commute with $C=\mu_1\cdot I$). We give however a more elementary argument. Indeed, we look for torsion matrices $A$ having eigenvalues $\lambda_1,\lambda_2$ that are roots of unity, such that $\mu_1-\lambda_1$ and $\mu_1-\lambda_2$ are also roots of unity. However, there are finitely many such roots of unity $\lambda_1$ and $\lambda_2$ if and only if $\mu_1\ne 0$. Indeed,  let $\lambda_1$ be a root of unity such that $\eta_1=\mu_1-\lambda_1$ is also a root of unity, that is, $\lambda_1+\eta_1=\mu_1$. If $\mu_1\ne 0$, 
there are at most two solutions $(\lambda_1,\eta_1)$ in roots of unity (and same discussion applies to $\lambda_2$) corresponding to the intersection of the two unit  circles in $\C$ centred at $0$ and $\mu_1$, proving (ii).  

Therefore, from now on, if $D$ is diagonal as above, we assume that $\mu_1\ne \mu_2$, and we proceed with proving (i).


\smallskip

$(\Longrightarrow)$ We assume first that $\Tr(C)$  is not the sum of at most two roots of unity, and we want to show that, up to conjugacy, there are finitely many torsion solutions to~\eqref{eq:xyd}. Let  $(A,B)$ be such a solution, and let $\Spec(A)=\{\lambda_1,\lambda_2\}$ and $\Spec(B)=\{\eta_1,\eta_2\}$. Taking the trace of~\eqref{eq:xyd}, we obtain the following equation in roots of unity:
\begin{equation}
\label{eq:eq roots 1}
\lambda_1+\lambda_2+\eta_1+\eta_2=\Tr(C).
\end{equation}

If $\Tr(C)$ is not an algebraic number,~\eqref{eq:eq roots 1} has no solution $(\lambda_1,\lambda_2,\eta_1,\eta_2)$ in roots of unity, therefore from now on we assume $\Tr(C)$ to be algebraic over $\Q$. Moreover, if $\Tr(C)=0$, then it is the sum of two roots of unity, say $1$ and $-1$, which contradicts our assumption. Therefore, we also have that $\Tr(C)\ne 0$.

We can apply now Lemma~\ref{lem:sum roots 1} to conclude that there are finitely many non-degenerate solutions $(\lambda_1,\lambda_2,\eta_1,\eta_2)$ in roots of unity to~\eqref{eq:eq roots 1}, that is, for which there is no vanishing subsum. Therefore such solutions lead to finitely many torsion matrices $A$ up to conjugacy and the same for $B$.

We consider now the possible vanishing subsums in~\eqref{eq:eq roots 1}, which up to symmetry, are as follows:
\begin{itemize}
\item[(i)] $\lambda_1=\Tr(C)$ and $\lambda_2+\eta_1+\eta_2=0$ (same discussion applies if $\lambda_1$  is replaced by $\lambda_2$; if $\lambda_1$ is replaced by any of $\eta_1$ or $\eta_2$, same discussion applies with $A$ interchanged with $B$).

\item[(ii)] $\lambda_1+\eta_1=\Tr(C)$ and $\lambda_2+\eta_2=0$ 
(same discussion applies if $\lambda_1+\eta_1$ is replaced by any combination $\lambda_i+\eta_j$ with $i,j\in \{1,2\}$).

\item[(iii)] $\lambda_1+\lambda_2=\Tr(C)$ and $\eta_1+\eta_2=0$ 
(if $\lambda_1+\lambda_2$  is replaced by $\eta_1+\eta_2$, same discussion applies with $A$ interchanged with $B$).
\end{itemize}

However, in all these cases we note that $\Tr(C)$ is the sum of at most two roots of unity, which contradicts our assumption. This concludes the proof of this implication.

\smallskip 

$(\Longleftarrow)$ We  assume now that the torsion solutions to~\eqref{eq:xyd} are contained in finitely many pairs of conjugacy classes. We want to show that $\Tr(C)$  is not the sum of at most two roots of unity. It is enough to construct examples when $\Tr(C)$ is a root of unity or the sum of two roots of unity that lead to infinitely many nonsimilar torsion matrices $A,B$ that satisfy~\eqref{eq:xyd}.

{\it Example1:} Assume $\Tr(C)$ is a root of unity. Let $\lambda_1,\lambda_2,\eta_1,\eta_2$ be roots of unity that satisfy
$$
\lambda_1=\Tr(C) \mand \lambda_2+\eta_1+\eta_2=0
$$
(that is, we are in case (i) above).

Therefore $\lambda_1$ is uniquely defined, and the second equation above implies that $$(-\eta_1/\lambda_2,-\eta_2/\lambda_2)\in \{e^{\pm \pi i/3},e^{\mp \pi i/3}\}.$$
However, if one varies $\lambda_2$ over all roots of unity, one can construct matrices $A$ with $\Spec(A)=\{\lambda_1,\lambda_2\}$ and $B$ with $\Spec(B)=\{\eta_1,\eta_2\}$ satisfying the above system. Indeed, let $\lambda_2\ne \lambda_1$ be any root of unity, $\eta_1=-e^{\pi i/3}\lambda_2$ and $\eta_2=-e^{-\pi i/3}\lambda_2$.

If $D$ in~\eqref{eq:D} is diagonal, the matrix
\begin{equation}
\label{eq:A1}
A=\begin{pmatrix}
\lambda_1+\lambda_2-d & d(\lambda_1+\lambda_2-d)-\lambda_1\lambda_2\\
1 & d
\end{pmatrix},
\end{equation}
where $$d=\frac{\lambda_2^2-\lambda_1\lambda_2-\mu_1\mu_2+(\lambda_1+\lambda_2)\mu_2}{\mu_2-\mu_1},$$ 
satisfies $\Spec(A)=\{\lambda_1,\lambda_2\}$ and  $B=C-A$ satisfies $\Spec(B)=\{\eta_1,\eta_2\}$. As $\lambda_2$ varies over all roots of unity, we obtain infinitely many matrices $A$ and $B$ which are not similar.

Similarly, if $D$ in~\eqref{eq:D} has the second Jordan form with eigenvalue $\mu$, then one can construct
$$
A=\begin{pmatrix}
2\mu & 0\\
\lambda_2^2+\mu^2-\mu \lambda_2 & \lambda_2
\end{pmatrix}.
$$
For such $A$ one has $\Spec(A)=\{\lambda_1,\lambda_2\}$ and  $B=C-A$ satisfies $\Spec(B)=\{\eta_1,\eta_2\}$.

\smallskip

{\it Example 2:} Assume $\Tr(C)$ is a sum of two roots of unity. If $\Tr(C)=0$ (which also falls in this case), then~\eqref{eq:eq roots 1} becomes
$$
\lambda_1+\lambda_2+\eta_1+\eta_2=0,
$$
which clearly has infinitely many torsion solutions $(\lambda_1,\lambda_2,\eta_1,\eta_2)$, and thus one can construct infinitely many non-similar torsion matrices $A$ and $B$.

We give now a construction when $\Tr(C)\ne 0$. Let $\lambda_1,\lambda_2,\eta_1,\eta_2$ be roots of unity that satisfy
$$
\lambda_1+\eta_1= \Tr(C) \mand \lambda_2+\eta_2=0
$$
(that is, we are in case (ii) above).

From the first equation $\lambda_1+\eta_1=\Tr(C)$, since  $\Tr(C)\ne 0$, we have at most two solutions $(\lambda_1, \eta_1)$ in roots of unity satisfying this equation. Let us fix one such torsion pair $(\lambda_1,\eta_1)$ such that $\lambda_1+\eta_1=\Tr(C)$.

However, as above, if one varies $\lambda_2$ over all roots of unity, one can construct matrices $A$ with $\Spec(A)=\{\lambda_1,\lambda_2\}$ and $B$ with $\Spec(B)=\{\eta_1,\eta_2\}$ satisfying the above system. Indeed,  let $\lambda_2\ne \lambda_1$ be any root of unity.

If $D$ in~\eqref{eq:D} is diagonal, let $A$ be defined by~\eqref{eq:A1}, where $$d=\frac{-\Tr(C)\lambda_2-\mu_1\mu_2+(\lambda_1+\lambda_2)\mu_2}{\mu_2-\mu_1}.$$ 
Then one has $\Spec(A)=\{\lambda_1,\lambda_2\}$ and  $B=C-A$ satisfies $\Spec(B)=\{\eta_1,\eta_2\}$. As $\lambda_2$ varies over all roots of unity, we obtain again infinitely many matrices $A$ and $B$ which are not similar.

Similarly, if $D$ in~\eqref{eq:D} has the second Jordan form with eigenvalue $\mu$, then one can construct
$$
A=\begin{pmatrix}
\lambda_1 & 0\\
-\mu^2+\mu(\lambda_1-\lambda_2) & \lambda_2
\end{pmatrix}.
$$
For such $A$ one has $\Spec(A)=\{\lambda_1,\lambda_2\}$ and  $B=C-A$ satisfies $\Spec(B)=\{\eta_1,\eta_2\}$, and thus one obtains infinitely many such non-similar matrices.

\smallskip

Similar construction can be made to create examples for the case (iii) above. This concludes the proof.

\subsection{Proof of Corollary~\ref{cor:Z^d+c}}
If $d=1$, then this is exactly one of the implications of the statement of Theorem~\ref{thm:lang deg 1} with $A$ replaced by $-U$ therein. Let $d\ge 2$. Then we look for torsion specialisations $U,B\in\GL_2(\C)$ such that
$$
-U^d+B=C.
$$
By Theorem~\ref{thm:lang deg 1}, there are finitely many pairs of conjugacy classes $(\cA,\cB)$ such that any torsion solution $(A,B)$ to $Z_1+Z_2=C$ belongs to one of these pairs, and thus, there are finitely many $U$, up to conjugacy, as well. This concludes the proof.

\section{Final comments}

We note that the strategy of the proof of Theorem~\ref{thm:lang deg 1} can also be used to study torsion solutions of the equation~\eqref{eq:xyd} in higher dimension $r\ge 3$ as well, however one would have to consider all possible vanishing subsums to a linear equation in $2r$ roots of unity, which becomes quickly very complicated. We would also like to extend Theorem~\ref{thm:LangMatr comm} and Corollary~\ref{cor:Z^d+c} to specialisations of arbitrary nonlinear polynomials $f,g\in\M_r(\C)$, obtaining a full analogue of Theorem~\ref{thm:Lang} for matrix polynomials.

\smallskip
We also ask a rather vague question towards obtaining a full matrix analogue of Theorem~\ref{thm:Lang} for torsion points on plane curves.

\begin{question}
\label{quest}
Let $F\in\M_r(\C)[Z_1,Z_2]$. Under what conditions on $F$ are there only finitely many pairs of conjugacy classes in $\GL_r(\C)^2$ containing all the  torsion solutions to the equation $F(Z_1,Z_2)=0$?
\end{question}

\smallskip
In Theorem~\ref{thm:LangMatr comm} and Theorem~\ref{thm:lang deg 1}, we look at matrices $A\in \M_r(\C)$ such that all the eigenvalues of $f(A)$ and $g(A)$ are roots of unity. 
It is also interesting to study the more general case when all the eigenvalues of $f(A)$ and $g(A)$ are of absolute value one, rather than just roots of unity. This, then, would be an analogue of Theorem~\ref{thm:PakShp} and would extend Corollary~\ref{cor:eigen abs val 1} to non-scalar matrices. 
\section*{Acknowledgement}
The author is very grateful to Julian Lawrence Demeio for noticing a gap in a main result of the previous version of the paper, and for follow-up discussions on this topic. The author also thanks Harry Schmidt for noticing this issue.

The author is also very grateful to Daniel Raoul Perez, Zeev Rudnick, Igor Shparlinski and Umberto Zannier for many discussions and useful comments on  preliminary versions of the paper. The author thanks Umberto Zannier for suggesting to consider Theorem~\ref{thm:lang deg 1}.

The author  was partially supported by the Australian Research Council Grant DP200100355. The author also gratefully acknowledges the hospitality and generosity of the Max Planck Institute of Mathematics (which fully supported the author for 11 months in 2020, during a very challenging Covid-19 time), where this project was initiated.


\begin{thebibliography}{99}

\bibitem{AR} N. Ailon and Z. Rudnick, `Torsion points on curves and common divisors of $a^k-1$ and $b^k-1$',  
{\it Acta Arith.\/},  {\bf 113} (2004),   31--38. 

\bibitem{BS} F. Beukers and C. J. Smyth, `Cyclotomic points on curves', 
{\it Number Theory for the
Millenium}, I (Urbana, Illinois, 2000),  A K Peters, 2002, 67--85.

\bibitem{BMZ99} E. Bombieri, D. Masser and U. Zannier, `Intersecting a curve with algebraic subgroups of multiplicative groups', {\it Int. Math. Res. Not.}, {\bf 20} (1999), 1119--1140.

\bibitem{BMZ08} E. Bombieri, D. Masser and U. Zannier, `On unlikely intersections of complex varieties with tori', {\it Acta Arith.}, {\bf 133} (2008), 309--323.

\bibitem{CMZ} P. Corvaja, D. Masser and U. Zannier, `Sharpening `Manin-Mumford' for certain algebraic groups of dimension 2', {\it Enseign. Math.}, {\bf 59} (2013), 1--45.

\bibitem{CoRuZa}  P. Corvaja, Z. Rudnick and U. Zannier,  `A lower bound for periods of matrices', 
{\it Comm. Math. Phys.}, {\bf  252} (2004), 535--541. 

\bibitem{CZ08}  P. Corvaja and U. Zannier, `On the maximal order of a torsion point on a curve in $\G_m^n$', {\it Rend. Lincei Mat. Appl.}, {\bf 19} (2008), 73--78.

\bibitem{DZ} R. Dvornicich and U. Zannier, `On sums of roots of unity', {\it Monatsh. Math.}, {\bf 129} (2000), 97--108.





\bibitem{Lang} S. Lang, \textit{Fundamentals of Diophantine Geometry}, Springer-Verlag, New
York, 1983.

\bibitem{Mann} H.B. Mann, `On linear relations between roots of unity', {\it Mathematika}, {\bf 12} (1965), 107--117.


\bibitem{Mau} G. Maurin, `Courbes alg\' ebriques et \' equations multiplicatives', {\it Math. Ann.}, {\bf 341} (2008), 789--824.




\bibitem{Ost} A. Ostafe, `On some extensions of the Ailon-Rudnick theorem', 
{\it Monat. f\" ur Math.\/},  {\bf 181} (2016), 451--471. 

\bibitem{PakShp}
F. Pakovich and I. E. Shparlinski, 
`Level curves of rational functions and unimodular points on rational curves', 
 {\it Proc. Amer. Math. Soc.\/},  {\bf  148} (2020),  1829--1833.

\bibitem{Zan95} U. Zannier, `Vanishing sums of roots of unity', {\it Rend Sem Mat Univ Pol Torino}, {\bf  53} (1995),  487--495.


 \bibitem{Zan} U.  Zannier, 
{\it Lecture notes on Diophantine analysis\/},
Publ. Scuola Normale Superiore, Pisa, 2009. 

\end{thebibliography}
\end{document}